\documentclass[12pt]{article}

 \usepackage[margin=10ex]{geometry}
\usepackage{hyperref}

\usepackage{amsfonts}
\usepackage{amsmath,amssymb}
\usepackage{graphicx,amscd,xypic}
\usepackage{bbding}
\usepackage{indentfirst}
\usepackage{mathrsfs, dsfont, bbm}
\usepackage[T1]{fontenc}

\newtheorem{thm}{Theorem}[section]

\newtheorem{lemma}[thm]{Lemma}

\newenvironment{proof}{\hspace{0ex}\textsc{Proof}.\hspace{1ex}}{\hfill$\Box$\\[2ex] }
\DeclareMathOperator{\dd}{\mathrm{d\!}}
\DeclareMathOperator{\BE}{\mathbf{E}}
\DeclareMathOperator{\BF}{\mathcal{F}}
\DeclareMathOperator{\BG}{\mathcal{G}}
\DeclareMathOperator{\BP}{\mathbf{P}}

\DeclareMathOperator{\BX}{\mathcal{X}}
\DeclareMathOperator{\BY}{\mathcal{Y}}
\DeclareMathOperator{\BZ}{\mathcal{Z}}

\DeclareMathOperator{\BH}{\mathcal{H}}

\DeclareMathOperator{\R}{\mathbb{R}}

\DeclareMathOperator*{\argmin}{\mathrm{argmin}}

\begin{document}
\title{A Note on the Monge-Kantorovich Problem in the Plane}
\author{ Zuo Quan Xu\footnote{Department of Applied
Mathematics, Hong Kong Polytechnic University, Hong Kong. This author
acknowledges financial supports from Hong Kong Early Career Scheme (No. 533112), Hong Kong General Research Fund (No. 529711) and Hong Kong Polytechnic University. Email: \url{maxu@polyu.edu.hk}.
 }\ \ and Jia-An
Yan\footnote{Academy of Mathematics and Systems Science, CAS, China.
This author acknowledges financial supports from National Natural
Science Foundation of China (No. 11371350), Key Laboratory of Random
Complex Structures and Data Science, CAS (No. 2008DP173182), and
Department of Applied Mathematics, Hong Kong Polytechnic University,
during his visit in December 2012. Email: \url{jayan@amt.ac.cn}. } }
\maketitle
\begin{abstract}
The Monge-Kantorovich mass-transportation problem has been shown to be fundamental for various basic problems in analysis and geometry in recent years. Shen and Zheng (2010) proposed a probability method to transform the celebrated Monge-Kantorovich problem in a bounded region of the Euclidean plane into a Dirichlet boundary problem associated to a nonlinear elliptic equation. Their results are original and sound, however, their arguments leading to the main results are skipped and difficult to follow. In the present paper, we adopt a different approach and give a short and easy-followed detailed proof for their main results.
\\[3mm]
\end{abstract}

\section{Introduction}
\noindent The optimal transportation problem was first raised by Monge in 1781. Let $X$ and $Y$ be two separable metric spaces, and
$c : X \times Y \to [0, \infty]$ be a Borel-measurable function, where $c(x,y)$ is the cost of the transportation from $x$ to $y$.
Given probability measures $\mu$ on $X$ and $\nu$ on $Y$, Monge's formulation of the optimal transportation problem is to find a
transport map $T : X \to Y$ that realizes the infimum $$\inf\left\{\int_X c(x,T(x))\dd \mu(x):\ T^{-1}(\mu)=\nu\right\},$$
where $T^{-1}(\mu )=\nu$ means that $\nu(A)=\mu(T^{-1}(A))$ for every Borel set $A$ on $Y$. Sometimes one can write $\mu T^{-1}=\nu$. A map $T$ that attains this infimum is called an optimal transport map. Monge's formulation of the optimal transportation problem can be ill-posed, because sometimes there is no $T$ satisfying $T^{-1}(\mu)=\nu$. Kantorovich (1942) reformulated this problem as following: to find a probability measure $\gamma$ on $X\times Y$ that attains the infimum
$$\inf\left\{\int_{X\times Y} c(x,y)\dd \gamma(x,y):\ \gamma\in \Gamma(\mu, \nu)\right\},$$
where $\Gamma(\mu, \nu)$ denotes the collection of all probability measures on $X \times Y$ with marginals $\mu$ on $X$ and $\nu$ on $Y$. It is known that a minimizer for this problem always exists when the cost function $c$ is lower semi-continuous and $\Gamma(\mu, \nu)$ is a tight collection of measures. Such optimization problem is called Monge-Kantorovich problem.
\par
We consider a special case: $X$ and $Y$ are both one-dimensional Euclidian domains, and $c(x,y)=|x-y|^2$. Let $(\Omega, {\cal F},
\BP)$ be a non-atomic probability space. We denote by ${\cal L}(F,G)$ the set of all 2-dimensional random variables whose marginal distributions are $F$ and $G$, respectively. Then the Monge-Kantorovich problem can be reformulated as follows: to find an optimal coupling of $(X,Y)\in{\cal L}(F,G) $ such that $\BE[(X-Y)^2]$ attains the minimum. It is well-known and easily proved (see, eg. Rachev and R\"uschendorf (1998\textsc{a},\textsc{b})) that if $(\widetilde{X},\widetilde{Y})\in {\cal L}(F,G)$, and $\widetilde{X}$ and $\widetilde{Y}$ are comonotonic\footnote{Two real-valued random variables $X$ and $Y$ are said to be comonotonic if $(X(\omega')-X(\omega)) (Y(\omega')-Y(\omega))\geqslant 0$ almost surely under $\BP\otimes\BP$.}, then $(\widetilde{X}, \widetilde{Y})$ is an optimal coupling:
$$(\widetilde{X}, \widetilde{Y})=\argmin_{(X,Y)\in{\cal L}(F,G)}\BE[(X-Y)^2],$$
and the minimum value is
$$\min_{(X,Y)\in{\cal L}(F,G)}\BE[(X-Y)^2]=\int_0^1|F^{(-1)}(t)-G^{(-1)}(t)|^2\dd t,$$
where $F^{(-1)}(\cdot)$ and $G^{(-1)}(\cdot)$ denote the left-continuous inverse functions of $F(\cdot)$ and $G(\cdot)$, respectively.

\section{Main Idea of Shen and Zheng (2010)}
\noindent In Shen and Zheng (2010), they consider the
Monge-Kantorovich problem in the Euclidean plane. Given two
2-dimensional distribution functions $\BF$ and $\BG$, they try to
find an optimal coupling of $(\BX,\BY)$ whose marginal distributions
are $\BF$ and $\BG$, respectively, such that
$$\BE[|\BX-\BY|^2]=\BE[(X_1-Y_1)^2+(X_2-Y_2)^2]$$ attains the
minimum, where $\BX=(X_1, X_2)$ and $\BY=(Y_1,Y_2)$. Denote
$\BZ=(X_1, Y_2)$, then
$$\BE[|\BX-\BY|^2]=\BE[(X_1-Y_1)^2+(X_2-Y_2)^2]=\BE[|\BZ-\BY|^2]+[|\BX-\BZ|^2].$$
Assuming that random vectors $\BX$, $\BY$ and $\BZ$ all have smooth and
strictly positive density functions, with help of the random vector
$\BZ$, Shen and Zheng (2010) have successfully reduced the dimension of the decision variable by turning the original optimal coupling
problem on $(\BX, \BY)$ into an optimization problem on the distribution
of $\BZ$. In Shen and Zheng (2010), they assume that the random vectors take
values in a bounded region and reformulate the problem in a new
probability space $(\widetilde{\Omega},\widetilde{ {\cal F}}, \widetilde{\BP})$ with
$\widetilde{\Omega}=[0,1]\times[0,1]$ and $\widetilde{\BP}$ being the
Lebesgue measure. In fact, as we will see below, this restriction and reformulation are not needed.
\par
The main approach in Shen and Zheng (2010) consists two steps. First, for each fixed pair $(X_1, Y_2)$, they adopt a probability approach to find the best $X_2$ and $Y_1$ to minimize $\BE[(X_2-Y_2)^2]$ and $\BE[(X_1-Y_1)^2]$ under the constraint that the joint distribution of $(X_1, X_2)$ is the given distribution $\BF$ and that of $(Y_1, Y_2)$ is $\BG$. After this step, the optimal coupling problem boils down to an optimization problem over all possible joint probability density functions of $(X_1, Y_2)$. They then propose a calculus of variations method to solve the above optimization problem.
\par
As in Shen and Zheng (2010), our first step is to construct two
functions $g(\cdot,\cdot)$ and $h(\cdot,\cdot)$ satisfying $(X_1, g(X_1,Y_2))\sim(X_1,X_2)$
and $(h(X_1,Y_2), Y_2 )\sim (Y_1,Y_2)$. Here $X\sim Y$
means that $X$ and $Y$ have the same distribution.
\par
 Let $f(\cdot,\cdot)$ be the probability density function of the 2-dimensional random vector
 $\BX=(X_1, X_2)$. Then the conditional distribution
of $X_2$ given $X_1=x$ is
$$F_{X_2|X_1}(y|x)=\BP(X_2\leqslant y|X_1=x)=\frac{\int_{-\infty}^y
f( x,t) \dd t}{\int_{\R}f(x,t)\dd t}.$$ For each fixed $x$, denote the
inverse function of $F_{X_2|X_1}(\cdot|x)$ by $G(x,\cdot)$, that
is $F_{X_2|X_1}(G(x,\cdot)|x)=\cdot$.
Let $p(\cdot,\cdot)$ be the probability density function of the 2-dimensional random vector
$\BZ=(X_1, Y_2)$. Then the conditional distribution
of $Y_2$ given $X_1=x$ is
$$ F_{Y_2|X_1}(y|x)=\BP(Y_2\leqslant y|X_1=x)=\frac{\int_{-\infty}^y p(
x,t) \dd t}{\int_{\R}p(x,t)\dd t}=\frac{\int_{-\infty}^y p( x,t) \dd
t}{\int_{\R}f(x,t)\dd t},$$ where the last identity is due to the
fact that
$$\int_{\R}p(x,t)\dd t=F'_{X_1}(x)=\int_{\R}f(x,t)\dd t.$$
Now define $g(x,y)=G(x, F_{Y_2|X_1}(y|x)).$
\par
Similarly, we define $h(x,y)=\widetilde{G}( F_{X_1|Y_2}(x|y),y)$,
where, for each fixed $y$, $\widetilde{G}(\cdot, y)$ is the inverse function of
$$F_{Y_1|Y_2}(x|y)=\BP(Y_1\leqslant x|Y_2=y)=\frac{\int_{-\infty}^x
\widetilde{f}(u,y) \dd u}{\int_{\R}\widetilde{f}(u,y)\dd u}$$ and
$\widetilde{f}(\cdot,\cdot)$ is the probability density function of the 2-dimensional random vector
$\BY=(Y_1, Y_2)$. And $F_{X_1|Y_2}(x|y)$ is given by
 $$F_{X_1|Y_2}(x|y)=\BP(X_1\leqslant x|Y_2=y)=\frac{\int_{-\infty}^x p(u,y) \dd u}{\int_{\R}p(u,y)\dd u}
=\frac{\int_{-\infty}^x p(u,y) \dd u}{\int_{\R}\widetilde{f}(u,y)\dd u},$$ where the last identity is due to the
 fact that
 $$\int_{\R}p(u,y)\dd u=F'_{Y_2}(y)=\int_{\R} \widetilde{f}(u,y)\dd u.$$
\par
Let $ \widehat{\BX}=(X_1, g(X_1,Y_2))$ and $
\widehat{\BY}=(h(X_1,Y_2),Y_2)$. In Shen and Zheng (2010), the authors
claim that $\widehat{\BX}\sim \BX$ and $\widehat{\BY}\sim \BY$ without
giving a proof. For the reader's convenience, we give a proof here.
\par
In fact, for any bounded Borel function $B$ on $\R^2$, we have
\begin{multline*}
\BE[B(X_1,g(X_1,Y_2)) ]=\int_{\R}\int_{\R} B(x,g(x,y))p(x,y)\dd x\dd y \\
=\int_{\R}\int_{\R} B(x,G(x, F_{Y_2|X_1}(y|x)))p(x,y)\dd x\dd y
=\int_{\R}\int_{\R} B\left(x,G\left(x, \frac{\int_{-\infty}^y p( x,t) \dd t}{\int_{\R}f(x,t)\dd t}\right )\right)p(x,y)\dd x\dd y.
\end{multline*}
Applying change of variable $v=\frac{\int_{-\infty}^y p( x,t) \dd t}{\int_{\R}f(x,t)\dd t}$, we have
$\dd v=\frac{ p( x,y) }{\int_{\R}f(x,t)\dd t}\dd y$ and
\begin{multline*}
\BE[B(X_1,g(X_1,Y_2)) ]
=\int_{\R}\int_{\R} B\left(x,G\left(x, \frac{\int_{-\infty}^y p( x,t) \dd t}{\int_{\R}f(x,t)\dd t}\right )\right)p(x,y)\dd x\dd y\\
=\int_{\R}\int_{0}^1B\left(x,G\left(x, v\right )\right)\int_{\R}f(x,t)\dd t \dd v\dd x.
\end{multline*}
Applying change of variable $u=G\left(x, v\right )$, we obtain
$F_{X_2|X_1}(u|x)=F_{X_2|X_1}(G\left(x, v\right )|x)=v$, $\dd v
=F'_{X_2|X_1}(u|x)\dd u=\frac {f( x,u)}{\int_{\R}f(x,t)\dd t} \dd
u$, and
\begin{multline*}
\BE[B(X_1,g(X_1,Y_2)) ]=\int_{\R}\int_{\R}B\left(x,G\left(x, v\right )\right)\int_{\R}f(x,t)\dd t \dd v\dd x\\
=\int_{\R}\int_{\R} B\left(x,u\right )f( x,u)\dd u\dd x=\BE[B(X_1, X_2) ].
\end{multline*}
This indicates $\widehat{\BX}\sim \BX$. Similarly, one can prove
$\widehat{\BY}\sim \BY$.
\par
In Shen and Zheng (2010), they claim that `` if $(\BX, \BY)$ is the
optimal coupling, then the above vector $(\widehat{\BX}, \widehat{\BY})$
have the same optimal joint distribution. Thus we have
$\BE[|\BX-\BY|^2]=\BE[|\widehat{\BX}-\widehat{\BY}|^2]$. '' Unfortunately,
we are not able to prove that $(\BX, \BY)$ and $(\widehat{\BX},
\widehat{\BY})$ have the same optimal joint distribution. Fortunately,
we will show that if $(\BX, \BY)$ is the optimal coupling, then we do
have
$\BE[|X-Y|^2]=\BE[|\widehat{\BX}-\widehat{\BY}|^2]$, that is to say, $(\widehat{\BX},
\widehat{\BY})$ is an optimal coupling as well.
\par
In fact, given $Y_2=y$, the conditional distributions of $X_1$ and $Y_1$ are $F_{X_1|Y_2}(\cdot|y)$ and $F_{Y_1|Y_2}(\cdot|y)$, respectively. Therefore,
\begin{align*}
 \BE[(X_1-Y_1)^2|Y_2=y]\geqslant \inf_{(X, Y)\in\mathcal{B}}\BE[(X-Y)^2],
\end{align*}
where $\mathcal{B}$ is the set of all 2-dimensional random vectors whose marginal distributions are
$F_{X_1|Y_2}(\cdot|y)$ and $F_{Y_1|Y_2}(\cdot|y)$, respectively. If $(\widetilde {X}, \widetilde {Y})\in\mathcal{B}$ and $\widetilde{X}$ and $\widetilde{Y}$ are comonotonic, then $(\widetilde {X}, \widetilde {Y})$ is an optimal coupling:
\begin{align*}
 \inf_{(X, Y)\in\mathcal{B}}\BE[(X-Y)^2]=\BE[(\widetilde{X}-\widetilde{Y})^2].
\end{align*}
It is an easy exercise to show that $(\widetilde {X}, \widetilde {Y})\in\mathcal{B}$ and $\widetilde{X}$ and $\widetilde{Y}$ are comonotonic if and only if $\widetilde{Y}=f(\widetilde{X})$,
where $$f(x)=F_{\widetilde{Y} }^{(-1)}(F_{\widetilde{X} }(x))=F_{Y_1|Y_2}^{(-1)}(F_{X_1|Y_2}(x|y)|y)=\widetilde{G}( F_{X_1|Y_2}(x|y),y)=h(x,y),$$
and $F_{\widetilde{Y} }^{(-1)}(\cdot)$ and $F_{Y_1|Y_2}^{(-1)}(\cdot|y)$ denote the left-continuous inverse functions of $F_{\widetilde{Y} }(\cdot)$ and $F_{Y_1|Y_2}(\cdot|y)$, respectively.
Therefore,
\begin{multline*}
 \BE[(X_1-Y_1)^2|Y_2=y]\geqslant \inf_{(X, Y)\in\mathcal{B}}\BE[(X-Y)^2]=\BE[(\widetilde{X}-\widetilde{Y})^2]=\BE[(\widetilde{X}-h(\widetilde{X}, y))^2]\\
=\BE[(X_1-h(X_1, y))^2|Y_2=y]=\BE[(X_1-h(X_1, Y_2))^2|Y_2=y],
\end{multline*}
where we used the fact that the conditional distribution of $X_1$
given $Y_2=y$ is the same as the distribution of $\widetilde{X}$,
that is $F_{X_1|Y_2}(\cdot|y)$. Now we obtain
\begin{align*}
\BE[(X_1-Y_1)^2]=\BE[\BE[(X_1-Y_1)^2|Y_2]]\geqslant \BE[\BE[(X_1-h(X_1,Y_2))^2| Y_2]]=\BE[(X_1-h(X_1,Y_2))^2].
\end{align*}
Similarly, one can prove that
\begin{align*}
\BE[(X_2-Y_2)^2]=\BE[\BE[(X_2-Y_2)^2|X_1]]\geqslant \BE[\BE[(g(X_1,Y_2)-Y_2)^2|X_1]]=\BE[(g(X_1,Y_2)-Y_2)^2].
\end{align*}
Adding them up, we get
\begin{multline*}
 \BE[|\BX-\BY|^2]=\BE[(X_1-Y_1)^2]+\BE[(X_2-Y_2)^2] \\
 \geqslant \BE[(X_1-h(X_1,Y_2))^2]+\BE[(g(X_1,Y_2)-Y_2)^2]=\BE[|\widehat{\BX}-\widehat{\BY}|^2].
\end{multline*}
Thus, we proved that if $(\BX, \BY)$ is an optimal coupling, so is $(\widehat{\BX}, \widehat{\BY})$.
\par
Note that
\begin{multline*}
\BE[|\widehat{\BX}-\widehat{\BY}|^2]=\BE[(X_1-\widetilde{G}(F_{X_1|Y_2}(X_1|Y_2), Y_2) )^2]+\BE[(G(X_1,F_{Y_2|X_1}(Y_2|X_1) )-Y_2)^2]\\
=\int_{\R}\int_{\R}\left(s-\widetilde{G}\left(\frac{\int_{-\infty}^s p(u,y) \dd u}{\int_{\R}\widetilde{f}(u,y)\dd u},y \right)\right)^2p(s,y)\dd s\dd y\\
+\int_{\R}\int_{\R}\left(t-G\left(x, \frac{\int_{-\infty}^t p(x,v)
\dd v}{\int_{\R}f(x,v)\dd v}\right ) \right)^2p(x,t)\dd t\dd x.
\end{multline*}
The optimal coupling problem in the Euclidean plane boils down to minimizing the right hand side of
the above identity over $\BH$, the set of all the probability density functions $p(\cdot,\cdot)$ satisfying
$\int_{\R}p(\cdot,t)\dd t=\int_{\R}f(\cdot,t)\dd t$ and
$\int_{\R}p(u,\cdot)\dd u=\int_{\R}\widetilde{f}(u,\cdot)\dd u$.
\par
In Shen and Zheng (2010), they propose a calculus of variations
method to solve the above optimization problem. However, their
arguments are skipped and difficult to follow. The main objective of
this note is to modify their method and give a detailed proof for
their main results.

\section{Solving the Problem: Calculus of Variations}
\begin{lemma}\label{lemma1}
If $\beta$ is continuous  in a neighbourhood of   $(a,b)\in\R^2$, then
\begin{align*}
 \lim_{ \varepsilon\to 0 } \frac{1}{\varepsilon^2}\int_{b}^{b+\varepsilon} \int_{a}^{a+\varepsilon}\beta(x,y)\dd x\dd y=\beta(a,b).
\end{align*}
\end{lemma}
\begin{proof}
This follows immediately from the mean value theorem.
\end{proof}
\begin{lemma}\label{mainlemma}
If $\beta$ is  second order continuously differentiable  in a neighbourhood of   $(a,b)\in\R^2$, then
\begin{align*}
\lim_{b_1\to b+}\lim_{a_1\to a+}\lim_{ \varepsilon\to 0+} \frac{1}{\varepsilon^2(a_1-a)(b_1-b)} \int_{\R} \int_{\R} \beta(x,y)\eta_{\varepsilon}(x,y)\dd x\dd y=\beta_{xy}(a,b),
\end{align*}
where
\begin{align}\label{eta}
 \eta_{\varepsilon}(x,y)=(\mathbf{1}_{[a_1, a_1+\varepsilon ]}(x)-\mathbf{1}_{[a, a+\varepsilon ] }(x))(\mathbf{1}_{[b_1, b_1+\varepsilon]}(y)-\mathbf{1}_{[b, b+\varepsilon]} (y)),\quad  a<a_1, \; b<b_1.
\end{align}
\end{lemma}
\begin{proof}
Note
\begin{align*}
 \int_{\R} \int_{\R} \beta(x,y)\eta_{\varepsilon}(x,y)\dd x\dd y=\int_{b}^{b+\varepsilon} \int_{a}^{a+\varepsilon}+\int_{b_1}^{b_1+\varepsilon} \int_{a_1}^{ a_1+\varepsilon}-\int_{b}^{b+\varepsilon} \int_{a_1}^{a_1+\varepsilon}-\int_{b_1}^{b_1+\varepsilon} \int_{a}^{a+\varepsilon} \beta\dd x\dd y.
\end{align*}
Applying Lemma \ref{lemma1} to each term above, we obtain
\begin{align*}
 \lim_{ \varepsilon\to 0+} \frac{1}{\varepsilon^2} \int_{\R} \int_{\R} \beta(x,y)\eta_{\varepsilon}(x,y)\dd x\dd y=\beta(a,b)+\beta(a_1,b_1)-\beta(a_1,b)-\beta(a,b_1).
\end{align*}
Therefore,
\begin{multline*}
\lim_{a_1\to a+}\lim_{ \varepsilon\to 0+} \frac{1}{\varepsilon^2(a_1-a) } \int_{\R} \int_{\R} \beta(x,y)\eta_{\varepsilon}(x,y)\dd x\dd y\\
=\lim_{a_1\to a+}\frac{\beta(a,b)+\beta(a_1,b_1)-\beta(a_1,b)-\beta(a,b_1)}{a_1-a}=\beta_x(a,b_1)-\beta_x(a,b),
\end{multline*}
which then follows
\begin{multline*}
\lim_{b_1\to b+}\lim_{a_1\to a+}\lim_{ \varepsilon\to 0+} \frac{1}{\varepsilon^2(a_1-a)(b_1-b)} \int_{\R} \int_{\R} \beta(x,y)\eta_{\varepsilon}(x,y)\dd x\dd y\\
=\lim_{b_1\to b+}\frac{\beta_x(a,b_1)-\beta_x(a,b)}{b_1-b}=\beta_{xy}(a,b).
\end{multline*}
The proof is complete.
\end{proof}
\par
Our objective is to minimize the functional $L(p)$ over $\BH$, where
\begin{multline*}
L(p)=\int_{\R} \int_{\R}\left(t-G\left(x,\int_{-\infty}^t \frac{p(x,v)}{f_1(x)}\dd v\right)\right)^2p(x,t)\dd t\dd x\\
+\int_{\R} \int_{\R}\left(s-\widetilde{G}\left( \int_{-\infty}^s \frac{p(u,y)}{f_2(y)}\dd u, y\right)\right)^2p(s,y)\dd s\dd y,
\end{multline*}
and
\begin{align*}
f_1(x)=\int_{\R}f(x,v)\dd v, \quad
f_2(y)=\int_{\R}\widetilde{f}(u,y)\dd u.
\end{align*}
\par
Suppose $p>0$ minimizes the functional $L(\cdot)$ over $\BH$. Let $\eta$ be any bounded function with compact support on ${\R}^2$ satisfying
\[\int_{\R} \eta(x,\cdot)\dd x=\int_{\R} \eta(\cdot,y)\dd y=0.\]
Then $p_{\varepsilon}:=p+\varepsilon \eta \in\BH$, when $\varepsilon$ is small enough.
Because $p_0=p$ minimizes the functional $L(\cdot)$ over $\{p_{\varepsilon}\}\cap\BH$, the first order condition reads
\begin{multline}\label{fristorder}
\frac{\partial}{\partial \varepsilon}\bigg[\int_{\R} \int_{\R}\left(t-G\left(x,\int_{-\infty}^t \frac{p_{\varepsilon}(x,v)}{f_1(x)}\dd v\right)\right)^2p_{\varepsilon}(x,t)\dd t\dd x \\
+\int_{\R} \int_{\R}\left(s-\widetilde{G}\left( \int_{-\infty}^s \frac{p_{\varepsilon}(u,y)}{f_2(y)}\dd u, y\right)\right)^2p_{\varepsilon}(s,y)\dd s\dd y\bigg]\bigg|_{\varepsilon=0}=0.
\end{multline}
Let us compute the first term in \eqref{fristorder},
\begin{align*}
&\hspace{3mm}\frac{\partial}{\partial \varepsilon}\bigg[\int_{\R} \int_{\R}\left(t-G\left(x,\int_{-\infty}^t \frac{p_{\varepsilon}(x,v)}{f_1(x)}\dd v\right)\right)^2p_{\varepsilon}(x,t)\dd t\dd x\bigg]\bigg|_{\varepsilon=0}\\
&=\int_{\R} \int_{\R}\frac{\partial}{\partial \varepsilon}\bigg[\left(t-G\left(x,\int_{-\infty}^t \frac{p_{\varepsilon}(x,v)}{f_1(x)}\dd v\right)\right)^2\bigg]\bigg|_{\varepsilon=0} p(x,t)\dd t\dd x\\
&\hspace{3mm}+\int_{\R} \int_{\R}\left(t-G\left(x,\int_{-\infty}^t \frac{p(x,v)}{f_1(x)}\dd v\right)\right)^2\eta(x,t)\dd t\dd x \\
&=\int_{\R} \int_{\R} 2\left(t-G\left(x,\int_{-\infty}^t \frac{p (x,v)}{f_1(x)}\dd v\right)\right) \left(-G_y\left(x,\int_{-\infty}^t \frac{p (x,v)}{f_1(x)}\dd v\right)\right)\left( \int_{-\infty}^t \frac{\eta (x,v)}{f_1(x)}\dd v \right) p(x,t)\dd t\dd x\\
&\hspace{3mm}+\int_{\R} \int_{\R}\left(t-G\left(x,\int_{-\infty}^t \frac{p(x,v)}{f_1(x)}\dd v\right)\right)^2\eta(x,t)\dd t\dd x \\
&=\int_{\R} \int_{\R}\left(\int_v^{\infty} 2\left(t-G\left(x,\int_{-\infty}^t \frac{p (x,v)}{f_1(x)}\dd v\right)\right) \left(-G_y\left(x,\int_{-\infty}^t \frac{p (x,v)}{f_1(x)}\dd v\right)\right) \frac{p(x,t) }{f_1(x)} \dd t\right) \eta (x,v)\dd v\dd x\\
&\hspace{3mm}+\int_{\R} \int_{\R}\left(t-G\left(x,\int_{-\infty}^t \frac{p(x,v)}{f_1(x)}\dd v\right)\right)^2\eta(x,t)\dd t\dd x \\
&=\int_{\R} \int_{\R}\varphi(x,y)\eta(x,y)\dd y\dd x,
\end{align*}
where
\begin{align*}
 \varphi(x,y)&=\int_y^{\infty} 2\left(t-G\left(x,\int_{-\infty}^t \frac{p (x,u)}{f_1(x)}\dd u\right)\right) \left(-G_y\left(x,\int_{-\infty}^t \frac{p (x,u)}{f_1(x)}\dd u\right)\right) \frac{p(x,t) }{f_1(x)} \dd t\\
 &\hspace{3mm}+\left(y-G\left(x,\int_{-\infty}^y \frac{p(x,u)}{f_1(x)}\dd u\right)\right)^2.
\end{align*}
Similarly,
\begin{align*}
 \frac{\partial}{\partial \varepsilon}\bigg[
\int_{\R} \int_{\R}\left(s-\widetilde{G}\left( \int_{-\infty}^s \frac{p_{\varepsilon}(u,y)}{f_2(y)}\dd u, y\right)\right)^2p_{\varepsilon}(s,y)\dd s\dd y\bigg]\bigg|_{\varepsilon=0} &=\int_{\R} \int_{\R} \psi (x,y)\eta(x,y)\dd y\dd x,
\end{align*}
where
\begin{align*}
 \psi(x,y)&=\int_x^{\infty} 2\left(s-\widetilde{G}\left( \int_{-\infty}^s \frac{p ( u,y)}{f_2(y)}\dd u,y\right)\right) \left(-\widetilde{G}_x\left(\int_{-\infty}^s \frac{p (u,y)}{f_2(y)}\dd u,y\right)\right) \frac{p(s,y) }{f_2(y)} \dd s\\
 &\hspace{3mm}+\left(x-\widetilde{G}\left( \int_{-\infty}^x \frac{p(u,y)}{f_2(y)}\dd u,y\right)\right)^2.
\end{align*}
Applying the first order condition \eqref{fristorder}, we deduce that
\begin{align*}
 \int_{\R} \int_{\R}( \varphi (x,y)+\psi(x,y))\eta(x,y)\dd y\dd x=0.
\end{align*}
Let us take $\eta\equiv \eta_{\varepsilon}$ defined in \eqref{eta} in the above equation. It then follows from Lemma \ref{mainlemma} that
\begin{multline*}
0=\lim_{b_1\to b+}\lim_{a_1\to a+}\lim_{ \varepsilon\to 0+} \frac{1}{\varepsilon^2(a_1-a)(b_1-b)} \int_{\R} \int_{\R}(\varphi (x,y)+ \psi(x,y))\eta_{\varepsilon}(x,y)\dd y\dd x\\={\varphi}_{xy}(a,b)+{\psi}_{xy}(a,b).
\end{multline*}
It is not hard to show
\begin{align*}
 {\varphi}_y(x,y)&=-2\left(y-G\left(x,\int_{-\infty}^y \frac{p (x,u)}{f_1(x)}\dd u\right)\right) \left(-G_y\left(x,\int_{-\infty}^y\frac{p (x,u)}{f_1(x)}\dd u\right)\right) \frac{p(x,y) }{f_1(x)} \\
 &\hspace{3mm}+2\left(y-G\left(x,\int_{-\infty}^y \frac{p(x,u)}{f_1(x)}\dd u\right)\right) \left(1-G_y\left(x,\int_{-\infty}^y \frac{p(x,u)}{f_1(x)}\dd u\right)\frac{p(x,y) }{f_1(x)}\right)\\
 &=2\left(y-G\left(x,\int_{-\infty}^y \frac{p(x,u)}{f_1(x)}\dd u\right)\right),
\end{align*}
and
\begin{align*}
 {\psi}_x(x,y)&=-2\left(x-\widetilde{G}\left( \int_{-\infty}^x \frac{p ( u,y)}{f_2(y)}\dd u,y\right)\right) \left(-\widetilde{G}_x\left(\int_{-\infty}^x \frac{p (u,y)}{f_2(y)}\dd u,y\right)\right) \frac{p(x,y) }{f_2(y)} \\
 &\hspace{3mm}+2 \left(x-\widetilde{G}\left( \int_{-\infty}^x \frac{p(u,y)}{f_2(y)}\dd u,y\right)\right) \left(1-\widetilde{G}_x\left( \int_{-\infty}^x \frac{p(u,y)}{f_2(y)}\dd u,y\right)\frac{p(x,y)}{f_2(y)}\right)\\
 &=2 \left(x-\widetilde{G}\left( \int_{-\infty}^x \frac{p(u,y)}{f_2(y)}\dd u,y\right)\right).
\end{align*}
Now we deduce the main theorem in Shen and Zheng (2010).
\begin{thm} Suppose $p>0$ minimizes $L(\cdot)$ over $\BH$. Denote $H(x,y)=\int_{-\infty}^x\int_{-\infty}^y p(u,v)\dd v\dd u$. Then
\begin{align*}
 \frac{\partial}{\partial x }\bigg[ G\left(x, \frac{1}{f_1(x)}H_x(x,y)\right) \bigg]
+\frac{\partial}{\partial y }\bigg[ \widetilde{G}\left( \frac{1}{f_2(y)} H_{y}(x,y),y \right) \bigg]=0.
\end{align*}
Moreover,
\begin{align*}
H(x,-\infty)&=0, & H(-\infty, y)&=0, \\
H (x,+\infty)&=\int_{-\infty}^xf_1(u)\dd u, & H (+\infty,
y)&=\int_{-\infty}^yf_2(v)\dd v.
\end{align*}
\end{thm}
\vskip 10pt
\noindent
{\bf Acknowledgement.} \ We thank Prof. Weian Zheng for valuable
discussions during the preparation of this note.



\end{document}